\documentclass[reqno]{amsart}
\usepackage{amsmath}
\usepackage{amssymb}
  \usepackage{paralist}
  \usepackage{graphics} %% add this and next lines if pictures should be in esp format
 \usepackage[colorlinks=true]{hyperref}
\usepackage{mathrsfs}
\usepackage{extarrows}
\usepackage{enumerate}

\newtheorem{theorem}{Theorem}[section]

\newtheorem{lemma}[theorem]{Lemma}
\newtheorem{proposition}[theorem]{Proposition}

\theoremstyle{definition}
\newtheorem{definition}[theorem]{Definition}

\newcommand{\ep}{\varepsilon}

\newcommand{\ZZ}{\mathbb{Z}}

\newcommand{\sC}{\mathscr{C}}

\newcommand{\sg}{\mathscr{G}}

\title[Reparametrized Gluing Orbit Property]
      {On the Entropy of Flows with Reparametrized Gluing Orbit Property}

\author[Peng Sun]{}


 \email{sunpeng@cufe.edu.cn}

\begin{document}

\maketitle\ 

% Enter the first author's name and address:
\centerline{\scshape Peng Sun}
\medskip
{\footnotesize
% please put the address of the first author
 \centerline{China Economics and Management Academy}
   \centerline{Central University of Finance and Economics}
   \centerline{Beijing 100081, China}
} % Do not forget to end the {\footnotesize by the sign }

\bigskip

%+Abstract
\begin{abstract}
We show that a flow or a semiflow with a weaker reparametrized form of gluing orbit property
is either minimal or of positive topological entropy.

\end{abstract}
%-Abstract

%+Contents
%\tableofcontents
%-Contents

\section{Introduction}

The gluing orbit property introduced in \cite{ST}, \cite{CT}
and \cite{BV} is a much weaker variation of the well-studied 
specification property. 
It is satisfied by a larger class of systems but is still productive,
as indicated by a series of recent works (see also \cite{BTV0},
\cite{BTV},
\cite{CLT}, \cite{TWW} and \cite{XZ}). 
A system with gluing orbit property may have zero topological
entropy, which is different from those with specification property.
However, it seems that such a system should be quite simple.
In \cite{Sun1}, we show that it must be minimal. In \cite{Sun1}
we have made  an effort to show that our results
holds in both discrete-time and continuous-time cases. In communication with
Paulo Varandas, we realize that such results should be more practical for flows if
a reparametrization is allowed. Then we have to overcome a bunch of 
technical difficulties. Finally, we are convinced that the result extends to
a even more general case.

\begin{theorem}\label{posent}
Let $(X,f)$ be a flow or semiflow with weak reparametrized gluing orbit property. 
Then either it is minimal, or
it has positive topological entropy.
\end{theorem}

\section{Preliminaries}
Let $(X,d)$ be a compact metric space. Denote by $f^t$ a flow or a semiflow on $X$.

\begin{definition}
For $L\ge1$, we call a strictly increasing continuous function %$\gamma:\RR\to\RR$ 
$\gamma: [0,\infty)\to[0,\infty)$ %for semiflows)
an \emph{$L$-reparametrization} if $\gamma(0)=0$ and
$$L^{-1}\le\frac{\gamma(t_1)-\gamma(t_2)}{t_1-t_2}\le L\text{ for any
$t_1,t_2\in[0,\infty)$, $t_1\ne t_2$}.$$
\end{definition}

\begin{definition}\label{gapshadow}
 We call the finite sequence of ordered pairs
$$\sC=\{(x_j,m_j)\in X\times[0,\infty): %x_j\in X, m_j\in\ZZ^+
%\text{ for each }%$j$, } 
j=1,\cdots, k\}$$
an \emph{orbit sequence of rank $k$}. A \emph{gap} for an orbit sequence
of rank $k$ is a $(k-1)$-tuple
$$\sg=\{t_j\in[0,\infty): j=1,\cdots, k-1\}.$$
Let $\gamma$ be a reparametrization.
For $\ep>0$, we say that $(\sC,\sg,\gamma)$ 
is be \emph{$\ep$-shadowed} by $z\in X$ if
for every $j=1,\cdots,k$,
$$(f^{\gamma(s_{j}+t)}(z), f^t(x_j))<\ep\text{ for every }t\in[0,m_j],$$
where
$$s_1=0\text{ and }s_j=\sum_{i=1}^{j-1}(m_i+t_i)\text{ for }j=2,\cdots,k.$$
\end{definition}

\begin{definition}\label{defrepglu}
We say that $(X,f)$  have 
\emph{reparametrized gluing orbit property},
 if for every $\ep>0$ there
is $M(\ep)>0$ such that for any orbit sequence $\sC$,
there are a \emph{gap} $\sg$ with $\max\sg\le M(\ep)$
and a $(1+\ep)$-reparametrization $\gamma$
such that 
%$$\max\{\gamma(s_{j+1})-\gamma(s_{j}+m_j):j=1,\cdots,k-1\}\le M(\ep)$$ 
%and 
$(\sC,\sg,\gamma)$  can be $\ep$-shadowed.
\end{definition}

%\begin{remark}
Definition \ref{defrepglu} is the equivalent to the definition of 
reparametrized gluing orbit property introduced in
\cite{BTV0}. % and \cite{BV}. 
It is natural to expect that most results
that hold for gluing orbit property also hold for reparametrized gluing,
where the reparametrization tends to identity as $\ep$ goes to $0$.
However, we perceive that our result only requires a weaker condition, where
the reparametrization can be %allowed to be 
more flexible when $\ep$ gets smaller.
%\end{remark}

\begin{definition}
We say that $(X,f)$  has
\emph{weak reparametrized gluing orbit property}
 if for every $\ep>0$ there
is $M=M(\ep)>0$ such that for any orbit sequence $\sC$,
there are a \emph{gap} $\sg$ with $\max\sg\le M$
and an $M$-reparametrization $\gamma$
such that 
%$$\max\{\gamma(s_{j+1})-\gamma(s_{j}+m_j):j=1,\cdots,k-1\}\le M(\ep)$$ 
%and 
$(\sC,\sg,\gamma)$  can be $\ep$-shadowed.
\end{definition}

It is clear that reparametrized gluing implies weak reparametrized
gluing. 

%The converse is not true. Suppose that $(X,g)$ is a homeomorphism
%with gluing orbit property and $r:X\to\RR^+$ is a continuous roof function.
%As $X$ is compact, $r$ is bounded and bounded away from $0$.
%Then the suspension flow has weak reparametrized gluing orbit property but
%may not have the stronger one.

%From now on we assume that $(X,f)$ is a non-minimal semiflow. 

\section{Proof of the Theorem}

Throughout this section, we assume that $(X,f)$ has 
weak reparametrized gluing property and it is not minimal.
We shall show that the topological entropy $h(f)>0$. 

\begin{lemma}\label{nonrecsf}
%Assume that $(X,f)$ has weak reparametrized gluing orbit property.
%Then 
There are $z\in X$, $\ep>0$ and $\tau>0$ such that
$$d(f^t(z),z)\ge\ep\text{ for any }t\ge\tau.$$
\end{lemma}

\begin{proof}
As $f$ is not minimal, there is a point whose orbit is not dense.
We can find $x,y\in X$ and $\delta>0$
such that
$$d(f^t(x),y)\ge\delta\text{ for any }t\ge0.$$
%Equivalently,
%$$\{f^t(x):t\ge 0\}\subset D_\delta,$$
%where $D_\delta=X\backslash B(y,\delta)$ is a compact set.

Let $0<\ep<\frac13\delta$ and $M:=M(\ep)$. 
For each $n\in\ZZ^+$, consider
$$\sC_n:=\{(y,0),(x,n)\}.$$
There are $\tau_n\in [0,M]$ and an $M$-reparametrization $\gamma_n$
 such that %$\gamma_n(\tau_n)\le m$ and
$(\sC_n,\{\tau_n\},\gamma_n)$ is $\ep$-shadowed by $z_n$.
This  implies that for any $t\ge M^2$ and $n\ge Mt$,
%there is $t_n=$
$$d(f^t(z_n),y)= d(f^{\gamma_n(t_n)}(z_n),y)
\ge d(f^{t_n}(x),y)-d(f^{\gamma_n(t_n)}(z_n),f^{t_n}(x))> 2\ep,$$
where 
$$t_n:=\gamma_n^{-1}(t)\in[M^{-1}t,Mt]\subset[\tau_n,\tau_n+n].$$

Let 
$z$ be a subsequential limit of $\{z_{n}\}$. Then %for $t\ge m^2$,
%$$d(f^{\tau+t}(x_0), f^{t}(x))\le\limsup_{n_k\to\infty}
%d(f^{\tau_{n_k}+t}(z_{n_k}), f^{t}(x))\le\ep\text{ for every }t\ge 0.$$
%and
$$d(f^{t}(z),y)\ge\liminf_{n\to\infty} d(f^t(z_n),y)
\ge2\ep\text{ for any $t\ge M^2$}.$$
Note that 
$$d(z,y)\le\liminf_{n\to\infty}d(z_n,y)\le\ep.$$ 
So for $\tau:=M^2$,
$$d(f^{t}(z),z)\ge d(f^{t}(z),y)-d(z,y)\ge
\ep\text{ for any $t\ge\tau$}.$$
\end{proof}

\begin{lemma}\label{stayawaysf}
%Assume that $(X,f)$ has weak reparametrized gluing orbit property.
%Then 
There are $x,y\in X$, $\ep>0$ and $T>0$ such that
\begin{align*}
%\min_{n\ge0}\min\{
d(f^t(x),x)&\ge\ep\text{ for any }t\ge T,\\
d(f^t(y),x)&\ge\ep\text{ for any }t\ge T,\\
 d(f^t(y),y)&\ge\ep\text{ for any }t\ge T,\text{ and }\\
 d(f^t(x),y)&\ge\ep\text{ for any }t\ge0.
\end{align*}
\end{lemma}

\begin{proof}
By Lemma \ref{nonrecsf}, there is $x\in X$, $\ep_0>0$ 
and $\tau>0$ such that
$$d(f^t(x),x)\ge\ep_0\text{ for every $t\ge \tau$}.$$ 
%Let $y\in B(x,\frac13\ep_0)$, $y\ne x$ ($x$ is not isolated). 
Let $\ep_1:=\frac13\ep_0$ and 
$M:=M(\ep_1)>1$. For each $n$, there are $\tau_n\in[0, M]$ and
an $M$-reparametrization $\gamma_n$ 
such that
$$(\{(x,M\tau),(x,n)\},\{\tau_n\},\gamma_n)$$ is 
$\ep_1$-shadowed by $y_n$.
Let 
$$T:=M(M\tau+M))>M\tau>\tau.$$
%This  implies that for any $t\ge m(mt_0+m)=:T$ 
For any $t\ge T$ and $n\ge Mt$,
%there is $t_n=$
$$d(f^t(y_n),x)= d(f^{\gamma_n(t_n)}(y_n),x)
\ge d(f^{t_n}(x),x)-d(f^{\gamma_n(t_n)}(y_n),f^{t_n}(x))> 2\ep_1,$$
where 
$$t_n:=\gamma_n^{-1}(t)\in[M^{-1}t,Mt]
\subset[M\tau+\tau_n,M\tau+\tau_n+n]
\text{ and }t_n>\tau.$$

As $\tau\le\gamma_n(M\tau+\tau_n)\le T$ for each $n$, there is a subsequence such
that
$$\gamma_{n_k}(M\tau+\tau_{n_k})\to T_0\ge \tau\text{ as }n_k\to\infty.$$
%Let $y_n=f^{t_n}(y_n')$. 
%Then
%$$d(f^{-t_n}(y_n),x)=d(y_n',x)<\ep_1<d(f^{-t_n}(x),x)$$
%and
%$$d(f^j(y_n), f^j(x))<\ep_1\text{ for }j=0,1,\cdots, n-1.$$
%As $t_n\in\{1,\cdots,m_1\}$ for every $n$, 
%There is a subsequence $\{\tau_{n_k}\}$ that converges to $\tau\in[0,m_1]$.
% such that $$\{t_n: t_n=t\}\text{ is infinite}.$$
%We can find a subsequence $\{y_{n_k}\}$ such that
%$$t_{n_k}=t\text{ for every }k.$$
Let $y$ be a subsequential limit of $\{y_{n_k}\}$. %Then
%\begin{equation}\label{followx}
%d(f^{t}(y), x)\le\ep_1\text{ for every }t\ge 0.
%\end{equation}
%This yields that for every $t\ge T$, 
Note that $d(x,y)<\ep_1$. So we have for any $t\ge T$:
\begin{align}
d(f^{t}(y),x)&
\ge \liminf_{n\to\infty} d(f^t(y_n),x)
\ge2\ep_1, \notag
%\label{awayx}
\\
d(f^{t}(y),y)&
\ge d(f^{t}(y),x)-d(x,y)
\ge %2\ep_1-\ep_1=
\ep_1,\text{ and}\notag\\
d(f^{t}(x),y)&\ge d(f^{t}(x),x)-d(x,y)\ge
%\delta-\ep_1=
2\ep_1.\label{awayy}
\end{align}
%Moreover
%$$d(f^{-t}(y),x)\le\limsup_{n_k\to\infty}d(y_{n_k}',x)\le\ep_1<d(f^{-t}(x),x),$$
%which 
 
For any $t\ge 0$,
$$d(f^{T_0}(y), x)\le\limsup_{n_k\to\infty}
d(f^{\gamma_{n_k}(M\tau+\tau_{n_k})}(y_{n_k}),x)
\le\ep_1<\ep_0\le d(f^{T_0+t}(x),x).$$
This guarantees that $f^t(x)\ne y$
for any $t\ge 0$,
Let 
$$\ep:=\min\{d(f^t(x),y):0\le t\le T\}.$$
%Then $0<\ep_2\le\ep_1$.
\iffalse and
$$d(f^t(x),y)\ge\ep_2\text{ for any }t\ge 0.$$
By the flow version of Lemma \ref{} (cf. \cite[]{CLT}), 
there is $\sg\in[0,M(\frac{\ep_2}2)]^\infty$ and $z\in X$ such that
$(\{(y,1)\}^{\infty},\sg)$ is $\frac{\ep_2}{2}$-shadowed by $z$.
This implies that the orbit of $z$ enters
$B(y,\ep_2)$ infinitely many times, which guarantees that
$$f^t(z)\ne x\text{ for any }t\ge 0.$$
\fi
Then $\ep\in(0,\ep_1)$. Together with \eqref{awayy} we have
$$d(f^t(x),y)\ge\ep\text{ for every }t\ge 0.$$
%The lemma holds for $x,y,\ep$ and $T=2t_0+\tau$.
\end{proof}

\begin{proposition}
$(X,f)$ has positive topological entropy.
\end{proposition}

%Now we complete the proof of Theorem \ref{posent}.
\begin{proof}
Let $x,y\in X$, $\ep>0$ and $T>0$ be as in Lemma \ref{stayawaysf}.
Let $0<\ep_0<\frac13\ep$ and $M:=M(\ep_0)>1$.
%We may assume that $\ep_2$ is so small that $m>T$. 
Denote 
$$T_2:=M^2(M+T)+T
\text{ and }T_1:=T+M^2(T_2+M)>T_2.$$ Let
$$Q_1:=\{(y,T_1)\}\text{ and }Q_2:=\{(x,T_2),(x,T_2)\}.$$

Let $n\in\ZZ^+$. For each 
$\xi=\{\omega_k(\xi)\}_{k=1}^n\in\{1,2\}^n$, consider
$$\sC_\xi:=\{Q_{\omega_k(\xi)}:k=1,\cdots, n\}=\{(x_j(\xi),m_j(\xi)):j=1,\cdots,n(\xi)\},$$
where
$$n(\xi)=\sum_{k=1}^n\omega_k(\xi).$$
There are $$\sg_\xi=\{t_j(\xi):j=1,\cdots,n(\xi)-1\}$$ %=\{t_j(\xi):j=1,\cdots,n-1\}$$ 
with $\max\sg_\xi\le M$ and an $M$-reparametrization $\gamma_\xi$ such
that $(\sC,\sg)$ is $\ep_0$-shadowed by $z_\xi\in X$.
For each $\xi$, denote
$$s_1(\xi):=0\text{ and }s_j(\xi):=\sum_{i=1}^{j-1}(m_i(\xi)+t_i(\xi))\text{
for }j=2,\cdots,n(\xi).$$
Then
$$s_{n(\xi)}(\xi)< nT_3\text{ for every }\xi\in\{1,2\}^n,$$
where $$T_3:=\max\{T_1+M,2T_2+2M\}.$$

We claim that if $\xi\ne\xi'$ then there is 
$$s\le\max\{\gamma_{\xi}(s_{n(\xi)}(\xi)),\gamma_{\xi'}(s_{n(\xi')}(\xi'))\}<nMT_3$$ 
such that
$$d(f^s(z_\xi),f^s(z_{\xi'}))>\ep_0.$$

Assume that $x_j(\xi)=x_j(\xi')$ for $j=1,\cdots,l-1$,
$x_l(\xi)=y$ and  $x_l(\xi')=x$.

 For $j<k$, denote
$$r_j:=\begin{cases}\gamma_{\xi}^{-1}(\gamma_{\xi'}(s_j(\xi')))-s_j(\xi),&\text{
if }\gamma_{\xi}(s_j(\xi))\le\gamma_{\xi'}(s_j(\xi'));\\
\gamma_{\xi'}^{-1}(\gamma_\xi(s_j(\xi)))-s_j(\xi'),&\text{
if }\gamma_{\xi}(s_j(\xi))>\gamma_{\xi'}(s_j(\xi')).
\end{cases}$$
Our discussion can be split into the following
cases.
\begin{enumerate}[\bf% {Case} 
(1)]
\item \emph{When $l=1$}.\\ Then
$$d(z_\xi,z_{\xi'})\ge d(x,y)-d(z_\xi,x)-d(z_{\xi'},y)>\ep_2.$$
We can take $s=0$.
\item \emph{When $l\ge2$ and there is $k\le l$ with $r_k\ge T$.}\\
 We may assume that $k$
is the smallest index with $r_k\ge T$. 
As $r_1=0$, we have $k\ge 2$ and hence
$r_{k-1}<T$. 
We assume that 
$\gamma_{\xi}(s_k(\xi))\le\gamma_{\xi'}(s_k(\xi'))$.
Argument for the subcase 
$\gamma_{\xi}(s_k(\xi))>\gamma_{\xi'}(s_k(\xi'))$
is analogous.
\begin{enumerate}[\bf %{Sub-case 
{(2.}1)]
\item \emph{When
$\gamma_{\xi}(s_k(\xi))\le\gamma_{\xi'}(s_{k-1}(\xi')+m_{k-1}(\xi'))$}.\\
Note that
\begin{align*}
\gamma_{\xi}(s_k(\xi))\ge{}&\gamma_{\xi}(s_{k-1}(\xi)+m_{k-1}(\xi))
\\\ge{}&
\gamma_{\xi}(s_{k-1}(\xi)+T_2)
\\\ge{}&
\gamma_{\xi}(s_{k-1}(\xi)+r_{k-1}+T_2-T)
\\\ge{}&
\gamma_{\xi}(s_{k-1}(\xi)+r_{k-1})+M^{-1}(T_2-T)\\
>{}&
\gamma_{\xi'}(s_{k-1}(\xi'))+MT\\
\ge{}&
\gamma_{\xi'}(s_{k-1}(\xi')+T).
\end{align*}
There is $r\in(T,m_{k-1}(\xi')]$ such that
$$\gamma_{\xi}(s_k(\xi))=\gamma_{\xi'}(s_{k-1}(\xi')+r).$$
Then
\begin{align*}
%\gamma_1:={}
&d(f^{\gamma_{\xi}(s_k(\xi))}(z_\xi), f^{\gamma_{\xi}(s_k(\xi))}(z_{\xi'}))\\
\ge{} &d(f^r(x_{k-1}(\xi')),x_k(\xi))-
d(f^{\gamma_\xi(s_k(\xi))}(z_\xi),x_k(\xi))
\\&-
d(f^{\gamma_{\xi'}(s_{k-1}(\xi')+r)}(z_{\xi'}),f^r(x_{k-1}(\xi')))\\
>{}&\ep_0.
\end{align*}
We can take $s=\gamma_{\xi}(s_k(\xi))$.
\item \emph{When
$\gamma_{\xi}(s_k(\xi))>\gamma_{\xi'}(s_{k-1}(\xi')+m_{k-1}(\xi'))$}.\\
We have
\begin{align*}
r_k={}&\gamma_{\xi}^{-1}(\gamma_{\xi'}(s_k(\xi')))-s_k(\xi)
\\\le{}&M((\gamma_{\xi'}(s_{k-1}(\xi')+m_{k-1}(\xi')+t_{k-1}(\xi'))
-\gamma_{\xi}(s_k(\xi)))
\\\le{}&M((\gamma_{\xi'}(s_{k-1}(\xi')+m_{k-1}(\xi')+M)
-\gamma_{\xi'}(s_{k-1}(\xi')+m_{k-1}(\xi')))
\\\le{}&M^3<T_2.
\end{align*}
Then $r_k\in(T,T_2)$ implies that
\begin{align*}
%\gamma_1:={}
&d(f^{\gamma_{\xi'}(s_k(\xi'))}(z_\xi), f^{\gamma_{\xi'}(s_k(\xi'))}(z_{\xi'}))\\
\ge{} &d(f^{r_k}(x_k(\xi)),x_k(\xi'))-
d(f^{\gamma_\xi(s_k(\xi)+r_k)}(z_\xi),f^{r_k}(x_k(\xi)))
\\&-
d(f^{\gamma_{\xi'}(s_k(\xi'))}(z_{\xi'}),x_k(\xi'))\\
>{}&\ep_0.
\end{align*}
\end{enumerate}
We can take $s=\gamma_{\xi'}(s_k(\xi'))$.

\item \emph{When $l\ge2$ and $r_{l}<T$}.
\begin{enumerate}[\bf {(3.}1)]
%\item\emph{When either
%$\gamma_{\xi}(s_l(\xi))\le\gamma_{\xi'}(s_{l-1}(\xi')+m_{l-1}(\xi'))$ or\\
%$\gamma_{\xi'}(s_l(\xi'))\le\gamma_{\xi}(s_{l-1}(\xi)+m_{l-1}(\xi))$ holds}.\\
%The argument is analogous to that for the case (2.1).
\item
\emph{When %$\gamma_{\xi}(s_{l-1}(\xi)+m_{l-1}(\xi))<
$\gamma_{\xi'}(s_l(\xi'))\le\gamma_{\xi}(s_l(\xi))$.}\\
Note that $r_l<T<T_2$.
We have
\begin{align*}
&d(f^{\gamma_\xi(s_l(\xi))}(z_\xi), f^{\gamma_\xi(s_l(\xi))}(z_{\xi'}))\\
\ge{} &d(y,f^{r_l}(x))-
d(f^{\gamma_\xi(s_l(\xi))}(z_\xi),y)-
d(f^{\gamma_{\xi'}(s_l(\xi')+r_l)}(z_{\xi'}),f^{r_l}(x))\\
>{}&\ep_0.
\end{align*}
We can take $s=\gamma_{\xi}(s_l(\xi))$.
\item
\emph{When %$\gamma_{\xi'}(s_{l-1}(\xi')+m_{l-1}(\xi'))<
$\gamma_{\xi}(s_l(\xi))<\gamma_{\xi'}(s_l(\xi'))$.}\\
Note that by the definitions of $\sC_\xi$ and $\sC_{\xi'}$, we must have
$$x_{l+1}(\xi')=x, m_{l}(\xi')=T_2\text{ and }n(\xi')\ge l+1.$$
Then in this case we have 
\begin{align*}
\gamma_{\xi'}(s_{l+1}(\xi'))
={}&\gamma_{\xi'}(s_{l}(\xi')+m_{l}(\xi')+t_l(\xi'))
\\\ge{}&
\gamma_{\xi'}(s_{l}(\xi'))+M^{-1}T_2\\
>{}&
\gamma_{\xi}(s_{l}(\xi))+MT\\
\ge{}&
\gamma_{\xi}(s_{l}(\xi)+T),\\
%\ge{}&
%\gamma_{\xi'}(s_{l}(\xi')+T).
%\end{align*}
%\begin{align*}
\gamma_{\xi'}(s_{l+1}(\xi'))
={}&\gamma_{\xi'}(s_{l}(\xi')+m_{l}(\xi')+t_l(\xi'))
\\\le{}&
\gamma_{\xi'}(s_{l}(\xi'))+M(T_2+M)\\
={}&
\gamma_{\xi}(s_{l}(\xi)+r_l)+M(T_2+M)\\
\le{}&
\gamma_{\xi}(s_{l}(\xi)+T+M^2(T_2+M)).
\\={}&
\gamma_{\xi}(s_{l}(\xi)+T_1).
\end{align*}
So there is $r\in(T,T_1]$ such that
$$\gamma_{\xi'}(s_{l+1}(\xi'))=\gamma_{\xi}(s_{l}(\xi)+r).$$
This yields that%We have
\begin{align*}
&d(f^{\gamma_{\xi'}(s_{l+1}(\xi'))}(z_\xi), f^{\gamma_{\xi'}(s_{l+1}(\xi'))}(z_{\xi'}))\\
\ge{} &d(f^r(y),x)-
d(f^{\gamma_\xi(s_l(\xi)+r)}(z_\xi),f^{r}(y))-
d(f^{\gamma_{\xi'}(s_{l+1}(\xi'))}(z_{\xi'}),x)\\
>{}&\ep_0.
\end{align*}
We can take $s=\gamma_{\xi'}(s_{l+1}(\xi'))$.
\end{enumerate}

\end{enumerate}

Above argument shows that 
$$E:=\{z_\xi: \xi\in\{1,2\}^n\}$$
is an $(nMT_3,\ep_0)$-separated subset of $X$
that contains $2^n$ points.
Hence
$$h(f)\ge\limsup_{n\to\infty}\frac{\ln s(nMT_3,\ep_0)}{nMT_3}\ge
\limsup_{n\to\infty}\frac{n\ln
2}{nMT_3}=\frac{\ln2}{MT_3}>0.$$

\end{proof}

\section*{Acknowledgments}
 We would like to thank Paolo
Varandas for suggestions and comments.
The author is supported by NSFC No. 11571387.

%+Bibliography

%-Bibliography

\end{document}